\documentclass{amsart}
\usepackage{fullpage}
\usepackage{amssymb}
\usepackage{amsthm}
\usepackage{amsmath}
\usepackage{color,epsfig}
\usepackage{graphicx}
\usepackage{mathdots}

\title{A Language Hierarchy and Kitchens-Type Theorem for Self-Similar Groups }

\author{Andrew Penland$^*$}
\address{*Corresponding Author. Department of Mathematics and Computer Science \\
Western Carolina University \\
Cullowhee, NC 28723 USA}
\email{adpenland@email.wcu.edu}

\author[Z. \v{S}uni\'{c}]{Zoran \v{S}uni\'{c}}
\address{Department of Mathematics \\
Hofstra University \\
Hempstead, NY 11549 USA}

\email{zoran.sunic@hofstra.edu}

\date{\today}

\DeclareMathOperator{\Aut}{Aut}
\DeclareMathOperator{\Sym}{Sym}
\DeclareMathOperator{\Stab}{Stab}

\DeclareMathOperator{\id}{id}
\DeclareMathOperator{\Triv}{Triv}

\DeclareMathOperator{\supp}{supp}

\newtheorem{theorem}{Theorem}

\newtheorem{cor}[theorem]{Corollary}
\newtheorem{proposition}[theorem]{Proposition}
\newtheorem{lemma}[theorem]{Lemma}

\newtheorem*{theoremA}{Theorem A}
\newtheorem*{theoremB}{Theorem B}

\theoremstyle{definition}
\newtheorem{example}[theorem]{Example}
\newtheorem{definition}[theorem]{Definition}

\begin{document}

\begin{abstract}
We generalize the notion of self-similar groups of infinite tree automorphisms to allow for groups which are defined on a tree but may not act faithfully on it. The elements of such a group correspond to labeled trees which may be recognized by a tree automaton (e.g. Rabin, B\"{u}chi, etc.), or considered as elements of a tree shift (e.g. of finite type, sofic) as in symbolic dynamics. We give examples to show how self-similar groups defined in this way can be separated into different tree language hierarchies. As the main result, extending the classical result of Kitchens on one-dimensional group shifts, we provide a sufficient condition for a self-similar group whose elements form a sofic tree shift to be a tree shift of finite type. As an application, we show that the closures of certain self-similar groups of rooted $k$-ary tree automorphisms that satisfy an algebraic law are not Rabin-recognizable, that is, they can not be described within the second order theory of $k$ successors. In both the main result and in the application a crucial role is played by a distinguished branched subgroup structure of the groups under consideration. 
\end{abstract}



\maketitle

\section{Introduction} 

In this paper, we consider regular, rooted trees with vertices labeled by elements of some finite alphabet. Our goal is to connect the computational, dynamical, and group-theoretic properties of these objects. As background and motivation for the present work, let us briefly describe some different perspectives on labeled trees, and how these perspectives relate to each other.

For our purposes, a {\em  $k$-regular, rooted tree} is an infinite directed graph with no non-oriented cycles which has a distinguished vertex called the {\em root}, such that each vertex has $k$ distinct {\em children}. (A vertex $v$ is a {\em child} of a vertex $w$ if there is an edge from $w$ to $v$.) Such a tree occurs as the {\em right Cayley graph} of a finitely generated free monoid, as follows. If $X$ is a finite set, then we write $X^*$ for the set of all finite words in $X$. This set forms a monoid with concatenation as the binary operation and the empty word $\epsilon$ as the identity. The right Cayley graph of this monoid has all words $w \in X^*$ as its vertices, with a directed edge from $v$ to $vx$ for all $v \in X^*$ and $x \in X$; this  graph clearly forms an $|X|$-regular tree whose root is the empty word $\epsilon$. 

When $|X| = 1$, the free monoid on $X$ is infinite cyclic and thus is isomorphic to the set of nonnegative integers under addition. The corresponding graph has the structure of a rooted, 1-regular tree, i.e. an infinite, one-way path where each vertex has exactly one child. We shall refer to this as the {\em one-dimensional case}. 

Given a fixed $k$-regular, rooted tree $\mathcal{T}$ and a finite alphabet $A$, one may consider the space of all $A$-labelings of $\mathcal{T}$, i.e. all possible functions from the vertices of $\mathcal{T}$ to $A$, which we denote by $A^{\mathcal{T}}$. Such spaces, consisting of labeled trees, appear in many different areas of mathematics; here we give a non-exhaustive list. 

$\bullet$ \textbf{Computation Theory}. In the theory of computation, a subset of $A^{\mathcal{T}}$ is called a {\em language}. There is a large body of work on these languages and the tree automata which accept them. In particular, the one-dimensional case of a one-way, infinite path has been extensively studied in theoretical computer science. Labelings of this path correspond to sequences or right-infinite words over the finite alphabet $A$, with finite connected subgraphs viewed as finite words over $A$. Hence, this is where one finds the classical notions of codes, languages, and various classes of automata which accept either finite or infinite words (such as in~\cite{NerodeAutomata}). 

$\bullet$ \textbf{Symbolic Dynamics.} Labeled trees also occur in modern symbolic dynamics, which is concerned with the properties of {\em shift spaces} on arbitrary groups or semigroups (as in ~\cite{CAGroups} and ~\cite{SymbolicHyperbolic}). If $T$ is a semigroup and $A$ is a finite alphabet, the {\em full shift} $A^T$ is defined as the space of all functions from $T$ to $A$. The full shift is equipped with the so-called {\em prodiscrete topology}, as well as a {\em shift action} $\sigma_t$ for each element $t \in T$. Each such shift action is a continuous self-mapping of the prodiscrete topological space $A^T$, and a subset $Y \subseteq A^T$ is called {\em shift-invariant} if $\sigma_t(y) \in Y$ for each $t \in T$ and $y \in Y$. A subset of the full shift $A^T$ is called a {\em shift space} if it is both shift-invariant and topologically closed as a subspace in the prodiscrete topology. 

A {\em pattern} $p$ is a function from a finite subset of $T$ to $A$. It is well-known that any shift space can be defined by declaring some collection of {\em forbidden patterns} which do not appear in any configuration in the space. If a finite collection of forbidden patterns may be used, the shift is called a {\em shift of finite type.} 

For any finite set $X$ there is a corresponding free monoid $X^*$ whose right Cayley graph has the structure of a tree. Hence, for a finite alphabet $A$, configurations of the full shift $A^{X^*}$ over this monoid correspond to the type of labeled trees which we consider. 

Again, in this area, the one-dimensional case has been extensively studied. The study of the connections between properties of one-dimensional shift spaces whose allowed or forbidden patterns form certain languages is well-established (see, for instance, the early chapters of ~\cite{LindMarcus}) and remains active. In the more general case of a rooted $k$-regular tree, Aubrun and Beal have examined algebraic and algorithmic aspects of various classes of tree languages(see~\cite{AubrunBeal1}, ~\cite{AubrunBeal2},~\cite{AubrunBeal3},~\cite{AubrunBeal2011}). Ceccherini-Silberstein, Coornaert, Fiorenza, and the second author have studied decision problems and automata related to {\em sofic tree shifts}~\cite{SoficTreeShiftCA}. 

$\bullet$ \textbf{Group Theory.} The labeled trees that we consider also arise in the area of {\em groups of tree automorphisms}. Such groups, including the well-known {\em Grigorchuk} group (introduced in~\cite{Grigorchuk-Burnside-1980}) and {\em Gupta-Sidki} group (introduced in~\cite{Gupta-Burnside-1983}), which are included in a class of groups now sometimes called {\em GGS}-groups. These groups have risen to prominence in the last few decades after they have been used as the solution to several decades-old open problems in group theory (see~\cite{Problems} for an overview for the first Grigorchuk group). Particularly interesting classes of groups in this area include {\em self-similar groups} (see the monograph~\cite{SelfSimilarGroups} for a comprehensive overview) and branch groups (a thorough overview of which is in ~\cite{BranchGroups}). 

If $X$ is a finite set, an automorphism of the right Cayley graph of $X^*$ is called an {\em $X^*$-automorphism}. The set of all $X^*$-automorphisms forms a group, which we denote by $\Aut(X^*)$. For each $k \geq 0$, the set $X^{[k]}$ consists of all words in $X^*$ whose length is at most $k$; this set inherits a subgraph structure from that of $X^*$, and we write $\Aut(X^{[k]})$ for the automorphisms of this subgraph. There is a natural group homomorphism $\pi_k: \Aut(X^*) \rightarrow \Aut(X^{[k]})$ given by restriction of the action of an automorphism to this finite subset. The kernel of $\pi_k$ consists of elements which fix each element of $X^{[k]}$ and is denoted by $\Stab(k)$. There is also a natural homomorphism $\Aut(X^{[k+1]}) \rightarrow \Aut(X^{[k]})$ for each $k \geq 0$, and this system of finite quotients and homomorphisms may be used to define $\Aut(X^*)$ as a {\em profinite topological group.} 

Each element $g \in \Aut(X^*)$ corresponds to a labeling of $X^*$ by the finite alphabet $\Sym(X)$; this labeling is called the {\em portrait of g}. Portraits provide a way to visualize the action of a particular automorphism, as well as a bijective mapping from the group $\Aut(X^*)$ to the full shift over $X^*$ with alphabet $\Sym(X)$. In Section 5 of ~\cite{Problems}, Grigorchuk observed that the portraits of the closure of the first Grigorchuk group in the profinite topology on $\Aut(X^*)$ formed a shift of finite type in the corresponding full shift. This led him to define a {\em group of finite type} as a group of tree automorphisms whose portraits form a shift of finite type when considered as a subshift of $(\Sym(X))^{X^*}$. 

Groups of finite type, also known as \textit{finitely constrained groups of tree automorphisms}, are characterized in the following theorem. 

\begin{theorem}\label{t:regular-branch}
Let $G$ be a level-transitive group of tree automorphisms of $X^*$ and $s \geq 0$. The following are equivalent.
\begin{enumerate}
\item[(i)] The group $G$ is the closure of some self-similar, regular branch group $H$, branching over the level $s$ stabilizer of $H$.
\item[(ii)] The group $G$ is a finitely constrained group defined by forbidden patterns of size $s + 1$.
\end{enumerate}
\end{theorem}

The direction (i) $\rightarrow$ (ii) was proven by the second author in ~\cite{HausdorffDim}, and the direction (ii) $\rightarrow$ (i) was shown earlier by Grigorchuk in ~\cite{Problems}. In ~\cite{PatternClosure}, the second author strengthened the result (ii) $\rightarrow$ (i) by proving that the group $H$ may be chosen to be countable. The proofs of these results are easily adapted to a more general definition of self-similar groups which we consider here; for the sake of completeness, we include a proof of the more general theorem in an appendix. 

Bondarenko and Samoilovych provided criteria to establish that a finitely constrained group is topologically finitely generated~\cite{Bondarenko}. In ~\cite{PenlandSunicHausDim}, the present authors applied these criteria to show that finitely constrained groups having a certain Hausdorff dimension were not topologically finitely generated. Additionally, the results of Fern{\' a}ndez-Alcober and Zugadi-Reizabal in~\cite{GGSFinitelyConstrained} show that if $G$ is a GGS-group acting on a $p$-regular tree where $p$ is an odd prime, and if $G$ has a non-constant defining vector, then $G$ is finitely constrained (this follows from combining the results of  Lemma 3.4 of ~\cite{GGSFinitelyConstrained} and Theorem~\ref{t:regular-branch}).   

The connection between group theory and subshifts in the one-dimensional case deserves mention as well.  If the finite alphabet $A$ is a group, then the one-sided, one-dimensional full shift over the alphabet $A$ gains a natural group structure isomorphic to the direct product $\displaystyle\prod_{n \in \mathbb{N}} A$ with elementwise group multiplication: $(gh)_n = g_nh_n$ for any $g, h \in A^{\mathbb{N}}$. An argument given by Kitchens in ~\cite{Kitchens87} for two-sided shifts over $\mathbb{Z}$ is easily adapted to one-sided shifts over $\mathbb{N}$ to give the following theorem. 

\begin{theorem}[Kitchens, 1987]\label{t:all-groups-shifts} 
Let $A$ be a finite alphabet which is also a group. If $G$ is a subset of $A^{\mathbb{N}}$ which is both a subshift and a subgroup of $A^{\mathbb{N}}$, then $G$ is a shift of finite type. 
\end{theorem}

\subsection*{Our results and organization of the paper} In this paper, we further investigate these connections between the computational, symbolic dynamics, and group-theoretic aspects of labeled trees. 

Section 2 contains the necessary details and background on the three areas outlined in the Introduction. For each action, not necessarily faithful, of a finite group $A$ on a finite set $X$, we consider a natural group structure on $A^{X^*}$ and its action on the regular rooted tree $X^*$, leading to a more general notion of a self-similar group. In this setting, self-similar groups are just subgroups of $A^{X^*}$ that are closed under the shift maps. Note that such a generalization is not just formally justifiable, but also necessary if one wishes to discuss classical one-sided one-dimensional group shifts in a common setting with all group tree shifts. Indeed, in the one-dimensional case, the tree is just a ray and the action of the group shift is necessarily trivial, while in higher dimensions the action could be faithful (as in the case of subgroups of $(\Sym(X))^{X^*}=\Aut(X^*)$). 

In Section 3, we present examples which show that various classes in the hierarchy of tree languages are in fact distinct, even when restricted to tree languages that are also subgroups of $A^{X^*}$. In particular, there are Rabin-recognizable (as tree languages) subgroups of $A^{X^*}$ that are not B\"{u}chi-recognizable and there are B\"{u}chi-recognizable subgroups of $A^{X^*}$ that are not sofic tree shifts. Further, we show that the closure of the odometer group is not finitely constrained, which shows that Theorem~\ref{t:all-groups-shifts} does not hold for shifts over arbitrary trees. However, we do not have an example of a self-similar group whose portraits form a sofic tree shift but do not form a tree shift of finite type. We further address the relationship between these two classes in Section 4. 

Our main result provides a sufficient condition for a self-similar group corresponding to a sofic tree shift to be a finitely constrained group.

\begin{theoremA}
Let $G$ be a subgroup of $A^{X^*}$. If the normalizer of $G$ in $A^{X^*}$ contains a self-replicating, level-transitive subgroup, then $G$ is a sofic tree shift group if and only if $G$ is a finitely constrained group.  
\end{theoremA} 

Note that in the one-dimensional case our sufficient condition is always satisfied by the trivial subgroup and, therefore, in that case the conclusion is valid without stating any conditions, which is in agreement with the original Kitchens Theorem (Theorem~\ref{t:all-groups-shifts}). 

We also note that the condition itself is of group-theoretic character, while the conclusion tells us something about the symbolic dynamics on the corresponding tree shift space. Such a deduction from one settings to another is, among other things, made possible by a more general version of Theorem~\ref{t:regular-branch} presented in the Appendix. Indeed, both Theorem~\ref{t:regular-branch} and its generalization Theorem~\ref{t:appendix} provide a group-theoretic description, a description using the subgroup structure, of the symbolic dynamics notion of a group tree shift of finite type. 

An even more interesting deduction from one setting to another occurs in our last result. 

\begin{theoremB}
Let $G$ be a self-similar, level-transitive group of tree automorphisms such that its normalizer in $\Aut(X^*)$ contains a self-replicating, level-transitive subgroup. If $G$ satisfies an algebraic law, or $G$ has a nontrivial center, then the topological closure $\overline{G}$ is not Rabin-recognizable. 
\end{theoremB}

The premise is, again, of group-theoretic character, but the conclusion concerns the (im)possibility of the description of the topological closure by using the second order theory of $k$ successors $SkS$ (for more details on second order theory of multiple successors see Rabin's work~\cite{RabinAutomata}, the survey~\cite{InfiniteAutomata}, or the book~\cite{NerodeAutomata}). While no notions from symbolic dynamics appear in either the premise or the conclusion, the connection between the group theory and logic is established in the proof through symbolic dynamics considerations. 

It should be noted that, referencing an earlier version of this work, Grigorchuk and Kravchenko showed that the portraits of the Lamplighter group do not form a sofic tree shift(~\cite{Grigorchuk-Lamplighter-Lattice}). 

\subsection*{Acknowledgements} The authors would like to thank David Carroll and Rostislav Grigorchuk for their comments on earlier versions of this work.

\section{Background}

\subsection{Tree Shifts and Unrestricted Rabin Automata} In this section, we review background related to symbolic dynamics and computation on labeled trees, giving rigorous definitions of some notions from the Introduction.  Much of the material in this section may be found in ~\cite{SoficTreeShiftCA} or ~\cite{InfiniteAutomata}. In some cases we will utilize the notation of ~\cite{KitchensBook} or ~\cite{CAGroups}. 

Let $X$ be a non-empty finite set. For $n \geq 0$, we define $X^n$ as the set of all words of length $n$ in $X$, writing $|w| = n$ to indicate $w \in X^n$. We let
 \[
  X^{(n)} = \bigcup_{i = 0}^{n-1} X^i , \qquad X^{[n]} = \bigcup_{i=0}^{n}  X^i \quad \text{ and } \quad X^* = \bigcup_{i = 0}^{\infty} X^i.
\] 

As we have previously seen, the elements of $X^*$ can be identified with the vertices of a regular rooted $|X|$-ary tree where the empty word $\epsilon$ is the {\em root}, and each vertex $w \in X^*$ has the set $\{ wx \}_{x \in X}$ as its {\em children}. The set $X^n$ is called \textit{level n} of the tree $X^*$, and the set $wX^* = \{ wu \; \mid \; u \in X^* \}$ is the \textit{subtree rooted at w}. A \textit{ray} $r$ in $X^*$ is a subset of $X^*$ that forms an infinite directed path beginning at $\epsilon$. 

Given a finite alphabet $A$, a \emph{configuration} is a map $f: X^* \rightarrow A$. The image of $w \in X^*$ under $f$ is called the \textit{label of f at w} and is denoted by $f_{(w)}$. The set $A^{X^*}$ of all configurations is called the \emph{full tree shift}. 

For $w \in X^*$, the \textit{shift map at w} is denoted $\sigma_{w} : A^{X^*} \rightarrow A^{X^*}$ and is defined by $(\sigma_w(f))_{(u)} = f_{(wu)}$. We abbreviate $\sigma_w(f)$ as $f_w$. (One must be careful not to confuse $f_w$ and $f_{(w)}$.) A subset $S$ of $A^{X^*}$ is called \textit{shift-invariant} if $\sigma_w(S) \subseteq S$ for all $w \in X^*$.

The set $A^{X^*}$ can be viewed as a compact metric space. Many alternative (but topologically equivalent) metrics are in use; a common choice is the metric $d$ given by $d(f,g) = 0$ if $f=g$, and 
\[ 
d(f,g) =  \frac{1}{2^{n}}, \text{where } n = \inf \{ \ k \in \mathbb{N} \; \mid \;  f _{\rvert X^{[k]}} \neq g_{\rvert X^{[k]}} \} 
 \] 
if $f \neq g$. The set of all rays forms a compact metric space called \textit{the boundary of $X^*$}, denoted $\partial X^*$. 

\begin{definition} A subset of a full tree shift which is topologically closed and shift-invariant is called a \emph{tree shift}.
\end{definition}

A \emph{pattern} is a map $p: M \rightarrow A$, where $M$ is some finite, non-empty subset of $X^*$. If the domain of a pattern $p$ is $X^{(n)}$ for some  $n$, we say that $p$ is a \textit{block of size n}. Given a configuration $f$ in $A^{X^*}$, a pattern $p$ with domain $M$, and some $w \in X^*$, we say that \textit{p appears in f at w} if $f_{w}\rvert _{M} = p$. Given a set $\mathcal{F}$ of \textit{forbidden patterns}, the set
 \[
 \mathcal{X}_{\mathcal{F}} = \{ f \in A^{X^*} \; \mid \;  \text{for all } p \in \mathcal{F},\; p \text{ does not appear in } f \} 
 \]
 defines a tree shift. Conversely, any tree shift $\mathcal{X}$ can be defined by declaring the forbidden patterns to be those that do not appear in any configuration of $\mathcal{X}$. Different sets of forbidden patterns may define the same tree shift. In the case that there exists a \textit{finite} set of forbidden patterns $\mathcal{F}$ such that $\mathcal{X} = \mathcal{X}_{\mathcal{F}}$, we say that $\mathcal{X}$ is a \textit{tree shift of finite type}. By extending patterns as necessary, we may assume that the forbidden patterns are all blocks of the same size. 

When $|X| = 1$, the shift $A^{X^*}$ may be viewed as the one-sided full shift over $\mathbb{N}$. In that case, a shift $\mathcal{X}$ is called \textit{sofic} if the blocks of $X$ form a regular language (accepted by some finite state automaton).  Analogously, we define sofic tree shifts to be those which can be accepted by a particular type of tree automaton. 

\begin{definition}
An \emph{unrestricted Rabin graph} is a 4-tuple $\mathcal{A} = (S,X,A, \mathcal{T})$ with $X$ and $A$ non-empty finite sets, $S$ a non-empty set, and $\mathcal{T}$ a subset of $S \times A \times S^{X}$. $X$ is called the \emph{tree alphabet}, $S$ is called the \emph{state set} or \emph{vertex set}, $A$ is called the \emph{label alphabet}, and $\mathcal{T}$ is called the \emph{set of transition bundles}.
\end{definition}

\begin{definition}
An \emph{unrestricted Rabin automaton} is an unrestricted Rabin graph with a finite state set. 
\end{definition}

To any configuration $f \in A^{X^*}$, we can associate an unrestricted Rabin graph $\mathcal{A}_f = (X^*,X,A,\mathcal{T}_f)$ with
 \[
 \mathcal{T}_{f} =  \{ (w; f_{(w)}; (wx)_{x \in X}) \; \mid \;  w \in X^* \}.
 \] 

Given two unrestricted Rabin graphs $\mathcal{A}_1 = (S_1,X,A,\mathcal{T}_1)$ and $\mathcal{A}_2 = (S_2, X, A, \mathcal{T}_2)$, a  \emph{homomorphism from $\mathcal{A}_1$ to $\mathcal{A}_2$} is a map $\alpha: S_1 \rightarrow S_2$ such that $(\alpha(s); a; (\alpha(s_x))_{x \in X}) \in \mathcal{T}_2$ whenever $(s; a; (s_x)_{x \in X}) \in \mathcal{T}_1$. We may also, in an overloading of notation, write the homomorphism $\alpha$ as $\alpha: \mathcal{A}_1 \rightarrow \mathcal{A}_2$. 

\begin{definition}
Let $\mathcal{A}$ be an unrestricted Rabin automaton. An element $f$ of $A^{X^*}$ is \emph{accepted} by $\mathcal{A}$ if there exists a homomorphism $\alpha_f: \mathcal{A}_f \rightarrow \mathcal{A}$. The \textit{language} $\mathcal{A}$ is the set of all configurations accepted by $\mathcal{A}$. A tree shift $\mathcal{X}$ is called \textit{sofic} if there exists an unrestricted Rabin automaton $\mathcal{A}$ such that $\mathcal{X} = \mathcal{X}_{\mathcal{A}}$.
\end{definition}

The class of sofic tree shifts is equal to the class of tree shifts which are the image of some tree shift of finite type under a continuous, shift-equivariant map called a \textit{cellular automaton} (see ~\cite[Theorem 1.7]{SoficTreeShiftCA} for the equivalence). From this it follows that the language of an unrestricted Rabin automaton $\mathcal{A}$ is a tree shift; we denote this shift by $\mathcal{X}_{\mathcal{A}}.$   Cellular automata are a very interesting and active area of study, but they will not enter into the rest of our discussion. 

\subsection{Other Classes of Tree Automata}

In this subsection, we introduce two additional classes of tree automata, \textit{B\"{u}chi automata} and \textit{Rabin automata}. They were initially introduced  (under different names) by Rabin in ~\cite{BuchiAutomata} and ~\cite{RabinAutomata}, respectively. (Note that what we call B\"{u}chi automata are also called \textit{special automata}.)

\begin{definition}
Let $A$ be a finite alphabet and $X$ be a finite, non-empty set. A B\"{u}chi automaton $\mathcal{B}$ (over $X$ with alphabet set $A$) is a 6-tuple $\mathcal{B} = (S,X,A,\mathcal{T},\mathcal{I},\mathcal{F})$ where $(S,X,A,\mathcal{T})$ is an unrestricted Rabin automaton, $\mathcal{I}$ is a non-empty subset of $S$ (called the \textit{set of initial states}) and $\mathcal{F} \subseteq S$ (called the set of \textit{accepting states}). 
\end{definition}

\begin{definition}
Let $A$ be a finite alphabet and $X$ be a finite, non-empty set. A Rabin automaton $\mathcal{R}$ (over $X$ with alphabet set $A$) is a 6-tuple $\mathcal{R} = (S,X,A,\mathcal{T},\mathcal{I}, \mathcal{F})$, where $(S,X,A,\mathcal{T})$ is an unrestricted Rabin automaton, $\mathcal{I}$  is a non-empty subset of $S$ (called the \textit{set of initial states}) and $\mathcal{F}$ (the \textit{set of accepting sets}) is a collection of subsets of $S$. 
\end{definition}

Acceptance in these classes of automata is based on the notion of a \textit{successful run}.

\begin{definition}
Let $f \in A^{X^*}$ be a configuration and $\mathcal{A} = (S,X,A,\mathcal{T},\mathcal{I}, \mathcal{F})$ be either a B\"{u}chi or Rabin automaton over $X$ with alphabet $A$. A \textit{run} of $\mathcal{A}$ on $f$ is a map $r:X^* \rightarrow S$ such that \begin{itemize}
\item $r$ is a homomorphism from the unrestricted Rabin automaton $\mathcal{A}_f$ to the unrestricted Rabin automaton $(S,X,A,\mathcal{T})$
\item $r(\epsilon) \in \mathcal{I}$
\end{itemize}
\end{definition}

For a configuration $f \in A^{X^*}$, a ray $\pi \in \partial X^*$, and a run $r$ of either a B\"{u}chi or Rabin automaton, let
$r_{\infty}(\pi,f) = \{ s \in S \; \mid \;  r^{-1}(s) \cap \pi  \text{ is infinite } \}$. A run $r$ of a B\"{u}chi automaton $\mathcal{B}$ on a configuration $f$ is called \textit{successful} if $ r_{\infty}(\pi,f) \cap \mathcal{F} \neq \emptyset
$,  for all $\pi \in \partial X^*$. For a Rabin automaton $\mathcal{R}$, a run $r$ is successful if for all $\pi \in \partial X^*$, there exists $F \in \mathcal{F}$ (which may depend on $\pi$) such that
$r_{\infty}(\pi,f) = F$. A configuration $f$ is \textit{accepted} by a Rabin (or B\"{u}chi) automaton $\mathcal{A}$ if there exists a successful run $r$ of $\mathcal{A}$ on $f$. 

For a Rabin (B\"{u}chi) automata $\mathcal{A}$, the \textit{language of} $\mathcal{A}$ is written as $\mathcal{L}(A)$ and defined as \[ \mathcal{L}(\mathcal{A}) = \{f \in A^{X^*} \; \mid \;  \text{ there exists a successful run } r \text{ of } \mathcal{A} \text{ on } f \}.\] A set $W \subseteq A^{X^*}$ is Rabin (B\"{u}chi) recognizable if there exists a Rabin (B\"{u}chi) automaton $\mathcal{A}$ such that $W = \mathcal{L}(\mathcal{A})$. 

It is a standard exercise to show that both Rabin and B\"{u}chi languages are closed under taking finite unions, finite intersections, and projections. It is well-known, but much more challenging to prove, that Rabin languages are closed under taking complements, while B\"{u}chi languages are not. 

Of course, a  B\"{u}chi tree automaton $\mathcal{B} = (S,X,A,\mathcal{T},\mathcal{I},\mathcal{F})$  may be seen as a Rabin tree automaton~$\mathcal{R} = (S,X,A,\mathcal{T},\mathcal{I}, \mathcal{F}')$  where $\mathcal{F}' = \{ F \subseteq S \; \mid \; F \cap \mathcal{F} \neq \emptyset \}$, so any B\"{u}chi recognizable set is Rabin recognizable. Similarly, any sofic shift is B\"{u}chi recognizable, since we may consider an unrestricted Rabin automaton as a B\"{u}chi automaton by taking all states to be both initial and final (initial and final states are not needed to accept a closed, shift-invariant set -- see ~\cite[Section 9]{SoficTreeShiftCA} for more details). Also, any tree shift of finite type is a sofic tree shift (see ~\cite[Theorem 1.9]{SoficTreeShiftCA}). 

\subsection{Group Structure}\label{ss:group-structure}

When $A$ is a finite group which acts on a set $X$, the elements in the infinite iterated wreath product $A \wr_X A \wr_X \ldots$ correspond to labeled trees in the full tree shift $A^{X^*}$. In this subsection, we give the details of this construction and discuss background from the case of tree automorphisms. 
The construction we give here is a generalization of {\em self-similar groups of tree automorphisms.} In that case, the alphabet is the group $\Sym(X)$, the full symmetric group on $X$, with its standard faithful action on $X$. Here, the infinitely iterated permutational wreath product
 \[ 
 \Sym(X) \wr_X \Sym(X) \wr_X \Sym(X)\wr  \ldots 
 \] yields a group of tree automorphisms of $X^*$. However, there is no reason why one must restrict the alphabet of tree shift to be a subgroup of $\Sym(X)$ -- there is a meaningful group structure on any tree subshift whose labels come from a finite group. . 

For any group $G$, we will write $e_G$ for the identity of $G$. 

Let $A$ be a group. Given a set $X$, a \textit{(left) group action} is a map $ A \times X \rightarrow X$, given by $(g,x) \rightarrow gx$ such that $ex= x$ and $g(hx) = (gh)x$ for all $g,h \in A$ and $x \in X$. Right actions are defined analogously. Given $x \in X$ and $g \in A$, we denote a left action of $g$ on $x$ by $g(x)$ and a right action of $g$ on $x$ by $x^g$.  If $B$ is a group, the set $B^{X} = \prod_{x \in X} B$ of all functions from $X$ to $B$ is a group with componentwise multiplication. The left action of $A$ on $X$ induces a right action of $A$ on $B^{X}$ given by 
\[
(b^{a})_{x} = b_{a(x)}. 
\] 

This action of $A$ on $B^{X}$ allows us to define the semi-direct product $A \ltimes B^{X}$, whose underlying set is $A \times B^{X}$ and where multiplication is given by
 \[ 
(a_1,b_1)(a_2,b_2) = (a_1a_2, b_1^{a_2}b_2) 
\] for $a_1,a_2 \in A$ and $b_1,b_2 \in B^{X}$. The group $A \ltimes B^{X}$ is called the \textit{permutational wreath product} of $(A,X)$ and $B$, and is denoted $A \wr_{X} B$.

In the case where $A,B$ are groups with left actions on the sets $X,Y$, respectively, the group $A \wr_{X} B$ has a left action on the set $X \times Y$, given by \[ (a,b)(x,y) = (ax,b_x{y}). \] 

If $X$ is a finite set and $A$ is a finite group with some left action $\phi$ on $X$, the permutational wreath product $A \wr_{X} A$ corresponds naturally to the set $A^{X^{(2)}}$ and acts on $X^2$ . The action of $A$ on $X$ naturally induces an action of $A$ on $(A \wr_{X} A)^{X}$, leading to the permutational wreath product \[ A \wr_{X} (A \wr_{X} A), \] which acts on $X^3$ and is identified with $A^{X^{(3)}}$. The groups $(A \wr_{X} A) \wr_{X} A$ and $A \wr_{X} (A \wr_{X} A)$ are canonically isomorphic, so we will omit parentheses. Iterating this construction $n$ times yields a group which acts on $X^{n}$. It is clear from the construction that the action of $\underbrace{A \wr A \wr \ldots \wr A}_{n \text{ times}}$ on $X^n$ preserves prefixes of words. Thus this group also has a well-defined action on $X^k$ for $0 \leq k \leq n$, which extends to an action on $X^{(n+1)}$. The infinitely iterated wreath product 
\[ 
A \wr_{X} A \wr_{X} A \ldots,
 \] 
naturally acts on $X^n$ for each $n \geq 1$, and thus it acts on $X^*$. The infinite iterated wreath product group whose construction we have just described will henceforth be denoted by $F(A,X,\phi)$ and called \textit{the full tree shift group} with alphabet $A$, action $\phi$ and space $X$. 

Note that \[ G = F(A,X,\phi) = A \wr_X A \wr_X A \ldots = A \wr_X (A \wr_X A \wr_X A \ldots) = A \wr_X G = A \ltimes G^X, \] and thus any element $g$ can be uniquely expressed as $(a, (g_x)_{x \in X})$, with $a \in A$ and $g_x \in G$.  For each $x \in X$, the element $g_x$ is the \textit{section of g at x}. We define $g_{\epsilon} = g$, and for a nontrivial word $w = vx \in X^*$ with $v \in X^*$ and $x \in X$, we recursively define $g_{vx} = (g_v)_x$.

Note that when writing the elements in the semi-direct product $A \ltimes G^{X}$, we often omit the outside parentheses and simply write the element as $a(g_x)_{x \in X}$, When $|X| = 2$, we write elements in the form $a(g_0,g_1)$.

\begin{definition}
A subgroup $H$ of $G = F(A,X,\phi)$ is \textit{self-similar} if $h_w \in H$ for all $h \in H$ and $w \in X^*$.
\end{definition}

There is a homomorphism $\alpha_{\epsilon}: F(A,X,\phi) \rightarrow A$ given, for $g = (a, (g_x)_{x \in X})$ by \[ \alpha_{\epsilon}(g) = a. \]  We call $\alpha_{\epsilon}(g)$ the \textit{root label of g}, and we define $\alpha_w: F(A,X,\phi) \rightarrow A$ by $\alpha_w(g) = \alpha_{\epsilon}(g_w)$. 

\begin{definition}
The \textit{portrait map} $\alpha: F(A, X, \phi) \rightarrow A^{X^*}$ is defined by \[ [\alpha(g)]_{(w)} = \alpha_w(g), \] and $\alpha(g)$ is called the \textit{portrait of g}. 
\end{definition}

The portrait map is a bijection, and whenever useful, we identify an element of $F(A,X,\phi)$ with its portrait in $A^{X^*}$. In particular, we endow $F(A,X,\phi)$ with a topology induced by its identification with $A^{X^*}$. This allows us to view subgroups of $F(A,X,\phi)$ as subspaces of $A^{X^*}$, and to study the properties of subgroups as spaces of portraits.  Often, if the action $\phi$ is understood or unimportant, we may simply write $F(A,X,\phi)$ as $A^{X^*}$. 

It is an important to note that the portrait of a section, $\alpha(g_v)$, corresponds to the shift of a portrait $\sigma_v(\alpha(g))$, so that a subgroup of $F(A,X,\phi)$ is self-similar if and only if its portraits form a shift-invariant subset of $A^{X^*}$. Thus, the portrait and section maps provide the correspondence between the group structure of the infinite iterated wreath product $F(A,X,\phi)$ and the structure of  $A^{X^*}$ as a shift space. \\

Again, it should be noted the self-similar groups we consider here are a generalization of the usual case of tree automorphisms. The following facts are well-known for self-similar groups of tree automorphisms, and also hold in the more general case discussed here. 

\begin{lemma}
Let $X$ be a finite set and $A$ be a finite group acting on $X$ via $\phi$. Let $g,h \in F(A,X,\phi)$, $x \in X$, and $u,v,w \in X^*$. 
Then the following hold.
\begin{enumerate}
\item $(gh)_w = g_{h(w)}h_w$
\item $(gh)_{(w)} = g_{(h(w))}h_{(w)}$
\item  $(g_{u})_{v} = g_{uv}$
\item $(g^{-1})_{(v)} = (g_{g^{-1}(v)})^{-1}$
\item Let $\supp(g) = \{ w \in X^* \; \mid \; g_{(w)} \neq e_A \}$. If $\supp(g) \subseteq wX^*$ and $\supp(h) \subseteq vX^*$ with $wX^* \cap vX^* = \emptyset$, then $gh = hg$. 
\end{enumerate}
\end{lemma} 

A self-similar group which is a tree shift of finite type (under the portrait map $\alpha$) is called a \textit{finitely constrained group}, and a self-similar group which is also a sofic tree shift (under the portrait map $\alpha$) is called a \textit{sofic tree shift group}. The reader should be warned that the term \textit{sofic} is used differently in group theory, as in ~\cite{SoficGroups}. However, we use the word \textit{sofic} to describe the tree shift, as in the traditional sense of symbolic dynamics. 

For the rest of this section, we write $A^{X^*}$ for $F(A,X,\phi)$. For $u \in X^*$ and a subgroup $H \leq A^{X^*}$, we define the stabilizer of $u$ as 
\[ \Stab_H(u) = \{ f \in A^{X^*} \; \mid \;  f(u) = u \}. 
\] 
and \textit{the level n stabilizer} (for $n \geq 0$) is the subgroup defined as 
\[ 
Stab_{H}(n) =  \{ f \in A^{X^*} \; \mid \;  f(u) = u \text{ for all } u \in X^{n} \} = \bigcap_{u \in X^{n}} \Stab_{H}(u). 
\] 

The subgroup $\Triv(n)$ of $A^{X^*}$ is given by 
 \[ 
\Triv(n) = \{ f \in A^{X^*} \; \mid \; f_{(u)} = e_A \text{ whenever } |u| < n \},
 \]
 $\Triv(n)$ is always a normal subgroup of $A^{X^*}$. For a subgroup $H \leq A^{X^*}$, we define 
\[    
 \Triv_{H}(n) = \Triv(n) \cap H. 
\] 

The group $\Triv_{H}(n)$ is a normal subgroup of $H$. Note that $\Triv_{H}(n)$ is always a subgroup of $\Stab_{H}(n)$; if the action of $A$ on $X$ is faithful, then the two groups coincide. 

\begin{definition}
A subgroup $H$ of $\Aut(X^*)$ is a \textit{regular branch group over} $\Stab_{H}(n)$ if it is level-transitive and, for all $h_0, \ldots, h_{k-1} \in \Stab_H(n)$, the element $(h_0,h_1, \ldots, h_{k-1})$ is also an element of $\Stab_{H}(n)$. 
\end{definition}

\begin{definition}
A subgroup $H$ of $A^{X^*}$ is a \textit{symbolic branch group, symbolically branching over} $\Triv_{H}(n)$ if, for all $h_0, \ldots, h_{k-1} \in \Triv_H(n)$, the element $(h_0,h_1, \ldots, h_{k-1})$ is also an element of $\Triv_{H}(n)$. 
\end{definition}

In the case of a faithful, level-transitive action, a group $H$ is a regular branch group over $\Stab_H(n)$ precisely when it is a symbolic branch group, symbolically branching over $\Triv_{H}(n)$. This distinction is important because level-transitivity is necessary for many results on branch groups of tree automorphisms, but will not be necessary for our more general definition of finitely constrained groups. 

For each $u \in X^*$ and subgroup $H \leq A^{X^*}$, there is a homomorphism $\phi_u: \Stab_{H}(u) \rightarrow A^{X^*}$ given by $g \mapsto g_u$. A self-similar group $H$ is \textit{self-replicating} if the map $\phi_u$ is surjective onto $H$ for all $u \in X^{*}$. In other words, a subgroup $H$ of $A^{X^*}$ is self-replicating if it is self-similar and for any $u \in X^*$ and any $g \in H$, we can find $h \in \Stab_{H}(u)$ such that $h_u = g$. A subgroup $H \leq A^{X^*}$ is \textit{level-transitive} if its action on $X^*$ is transitive on each $X^n$.

%=================================================================

\section{Tree Languages for Portraits of Self-Similar Groups}

In this section, we consider the portraits of self-similar groups as tree languages recognized by various classes of infinite tree automata. In general, the different classes of configuration subspaces discussed so far form a hierarchy as follows: 
\[ \text{SFT} \subsetneq \text{SOFIC} \subsetneq \text{B\"{U}CHI RECOGNIZABLE} \subsetneq \text{RABIN RECOGNIZABLE}. \]

Our goal is to explore this hierarchy for the portrait spaces of self-similar groups. Henceforth, whenever $A$ is a finite group which has a left action $\phi$ on a finite set $X$, we will write the iterated wreath product $F(A,X,\phi)$ as $A^{X^*}$ (suppressing reference to $\phi$).

We begin by showing that there are B\"{u}chi-recognizable tree shift groups which are not sofic tree shift groups. 

\begin{example}[A B\"{u}chi-recognizable self-similar group which is not a sofic tree shift group] Let $G = A^{X^*}$ be a full tree shift group for some finite set $X$, some finite group $A$, and some left action $\phi$. Let $B$ be a proper subgroup of $A$. We define the subset $H_B$ to be \[  H_B = \{ h \in A^{X^*} \mid \text{ there exists an }N_h \text { such that } |v| > N_h \text{ implies that } h_{(v)} \in B \}. \]  When $B$ is the trivial group and the action $\phi$ is faithful, this corresponds to the group of \textit{finitary tree automorphisms}(as discussed in ~\cite{Brunner-Sidki-1997}) . Note that $H_B$ is self-similar. We will show that $H_B$ is a subgroup of $G$. If $g, h \in H_B$, then there exist $N_1,N_2$ such that $g_{(v)} \in B$ whenever $|v| > N_1$, and  ${h^{-1}}_{(v)} \in B$ whenever $|v| > N_2$. Then, taking $N = \max\{N_1,N_2\}$, it follows that whenever $|v| > N$, we have 
\[ 
{gh^{-1}}_{(v)} = {g}_{\left(h^{-1}(v)\right)}({h}_{({h^{-1}(v)})})^{-1} \in B,
 \]
 as $|h^{-1}(v)| = |v|$. Also, it is clear that $H_B$ is self-similar and self-replicating. Moreover, if the action of $B$ on $X$ is transitive, then $H_B$ is level-transitive, as well. However, $H_B$ is not closed. In fact, $H_B$ is dense in $G$, since for any $g \in G$, we can obtain an element $h \in H_B$ which is arbitrary close to $g$ by by choosing the appropriate $n$ and setting $h_{(w)} = g_{(w)}$ for $w \in X^{[n]}$, $h_{w} = e_A$ if $|w| > n$ . Since $H_B$ is not topologically closed, it is not sofic. 

We will show that $H_B$ is B\"{u}chi. Consider the B\"{u}chi automaton $\mathcal{B} = (S,X,A,\mathcal{T},\mathcal{I},\mathcal{F})$, where \begin{itemize}
\item $S = \{ s_1, s_2 \} $
\item $\mathcal{T}$ consists of transition bundles of the following forms:
\subitem for all $a \in A$, $T_{1,a} = (s_1; a; {(s_1)}_{x \in X}) \in \mathcal{T}$;
\subitem for all $b \in B$, $T_{1,2,b}  = (s_1; b; {(s_{2})}_{x \in X}) \in \mathcal{T}$;
\subitem for all $b \in B$, $T_{2,b} = (s_2; b; {(s_{2})}_{x \in X}) \in \mathcal{T}$
\item $\mathcal{I} = \{ s _1 \}$
\item $\mathcal{F} = \{ s_2 \}$
\end{itemize}

The computation of $\mathcal{B}$ can be described as follows. The automaton begins in the initial state $s_1$. It can remain in $s_1$ by reading any element of $A$, or it can transition to $s_2$ by reading any element of $B$. Once in $s_2$, it will remain at $s_2$, at which point it can only read elements of $B$. Thus a configuration is accepted by $\mathcal{B}$ if there exists a run which eventually reaches $s_2$ along every ray and never leaves that state. We claim that the set of elements accepted by $\mathcal{B}$ is the same as the subgroup $H_B$.

If $h \in H_B$, with $N_h$ such that $h_{(v)} \in B$ whenever $|v| > N_h$, we can define a successful run $r_h: \mathcal{A}_h \rightarrow \mathcal{B}$ by 
\[
 r_h(v) = \begin{cases} s_1, &|v| \leq N_h \\ s_2, &|v| > N_h \end{cases}.
\] 
Thus $H_B \subset \mathcal{L}(\mathcal{B})$.

Now suppose $h \in \mathcal{L}(\mathcal{B})$. Let $r_h$ be a successful run of $h$ on $\mathcal{B}$. By the definition of B\"{u}chi acceptance, each ray $\pi$ must take the value $s_2$ infinitely often. However, every transition bundle in $\mathcal{B}$ which begins at $s_2$ also ends at $s_2$, so $r_h(w) = s_2$ implies that $r_h(wx) = s_2$ for all $x \in X$. It follows from induction that for every ray $\pi$, there is a $w \in \pi$ such that $r_h$ only takes the value $s_2$ on $wX^*$. Since each ray $\pi$ is contained in some such $wX^*$, and the sets $wX^*$ are open in the topology of $\partial X^*$, these sets form an open cover of $\partial X^*$. Since $\partial X^*$ is compact, we can take a finite collection $\{ w_1, w_2, \ldots, w_n \}$ such that each ray $\pi$ is contained in at least one open set from the finite collection $\{ w_iX^* \}_{i=1}^n$ and such that for each $v \in w_iX^*$, $r_h(v) = s_2$. Taking
 \[
N = \max \{ |w_1|, |w_2|, \ldots, |w_n| \}, 
\]
 we have that $r_h(v) = s_2$ whenever $|v| > N$. However, the only transition bundles from $s_2$ to itself are labeled by elements of $B$, so we have that $v_{(h)} \in B$ whenever $|v| > N$. Thus $h \in H_B$.
\end{example}

The next example utilizes the standard construction of a Rabin recognizable subset which is not B\"{u}chi recognizable. The key observation is that this tree language describes the portraits of a self-similar group.

\begin{example}[A Rabin-recognizable tree shift group which is not B\"{u}chi-recognizable]
Let $X = \{0,1\}$ and $A = C_2 = \{  \id, \sigma \} $ be the cyclic group of order 2 acting transitively on $X$. Let 
 \[
 H = \{ g \in  A^{X^*} \mid \text{ every ray in } g \text{ has only finitely many vertices with nontrivial label}  \}.
\] 

The set $H$ is a well-known example of a tree language which is Rabin but not B\"{u}chi (see ~\cite{InfiniteAutomata}). It is clear that $H$ is self-similar. 

Notice that $H$ is a subgroup of $A^{X^*}$, since an element $h \in A^{X^*}$ is in $H$ if and only if for every ray $\pi$, there exists an $N$ such that for all $v \in \pi$ with $|v| > N$, $h_{(v)} = \id$.  Let $g,h \in H$, and let $\pi$ be a ray in $gh^{-1}$. Since $g \in H$ and $h^{-1}(\pi)$ is a ray in $X^*$, there exist an $N_1$ such that $(g)_{(h^{-1}(v))}$ is the identity whenever $|h^{-1}(v)| = |v| > N_1$. Since $h \in H$, there exists an $N_2$ such that $h_{(v)}$ is the identity whenever $|v| > N_2$. Taking $N = \max \{ \ N_1,N_2 \ \}$, it follows that whenever $|v| > N$  \[ 
{gh^{-1}}_{(v)} = {g}_{\left(h^{-1}(v)\right)}({h}_{({h^{-1}(v)})})^{-1} = \id,
 \]Thus $H$ is a self-similar, self-replicating, level-transitive subgroup which is Rabin-recognizable but is not B\"{u}chi-recognizable. 
\end{example}

The next example shows that there exist tree shift groups over $\{0,1\}^*$ which are not finitely constrained groups, in contrast to the one-dimensional case considered by Kitchens(\cite{Kitchens87}). 

\begin{example}[A tree shift group which is not finitely constrained]
Let $X = \{0,1\}$, $A = \Sym(X) = \{\id, \sigma \}$, and let $A$ act faithfully on $X$ by permutations, so that the group $A^{X^*} = \Aut(X^*)$. We use $\id$ to represent the identity of $C_2$ so that $e$ can be reserved here for the identity of $\Aut(X^*)$.  Let $a = \sigma (e,a) \in \Aut(X^*)$. In terms of labels, 
\[
 a_{(w)} = 
  \begin{cases}
   \sigma, & \text{if } w = 1^{n} \text{ for some } n \geq 0 \\ 
   \id, &\text{otherwise }. 
  \end{cases}. 
\] 

Any section of $a$ is either the identity or $a$, so the group generated by $\{e, a \}$ is self-similar. This group $\mathcal{O}$ is called the \textit{odometer group}, as it ``rolls over'' any word consisting of all 1's. We claim that $\overline{\mathcal{O}}$ is not a finitely constrained group.  The proof uses the structure of the portraits of $\mathcal{O}$. 

It follows by induction that 
\[ 
 a^{2n} = \id(a^n,a^n) 
\] 
and 
\[ 
 a^{2n+1} = \sigma(a^{n},a^{n+1}). 
\] 
We note that $a^{2n}_{(\epsilon)} = \id$, while $a^{2n+1}_{(\epsilon)} = \sigma$. We also observe that $a^{2^n} \in \Triv(n)$ for all $n \in \mathcal{N}$. 

Since $\mathcal{O}$ is self-similar, the closure $\overline{\mathcal{O}}$ is a tree shift group, We will show that $\overline{\mathcal{O}}$ is not finitely constrained. Suppose that $\overline{\mathcal{O}}$ is finitely constrained by some set $\mathcal{F}$ of forbidden blocks of size $n+1$. Consider the element $g \in A^{X^*}$ with root label $g_{\epsilon} = \id$ and sections given by $g_0 = e_{G}$, $g_1 = a^{2^{n}}$. Each pattern of size $n+1$ in $g$ is either a pattern in $e_G$ or a pattern in $a^{2^n}$, so $g$ must be in the shift space $X_{\mathcal{F}}$. 

\noindent Since we assumed $X_{\mathcal{F}} = \overline{\mathcal{O}}$, there must be a sequence of elements in $\mathcal{O}$ which converge to $g$. Since $g_{\epsilon} = \id$, this sequence must eventually consist of even powers of $a$, so \[
a^{2n_i} \rightarrow g.
\]

Then $(a^{2n_i})_{0} \rightarrow g_0$ and $(a^{2n_i})_1 \rightarrow g_1$. However, $(a^{2n_i})_0 = (a^{2n_i})_1$, but $g_{0} \neq g_{1}$. Therefore we have a contradiction, and $\overline{\mathcal{O}}$ is not finitely constrained. 

\end{example}

This result can also be shown by first proving a stronger result that relies on the structure of regular branch groups. 

\begin{proposition}\label{p:algebraic-law-not-finitely-constrained}
Let $G$ be a level-transitive group of tree automorphisms such that $G$ satisfies an algebraic law or has a nontrivial center. The topological closure $\overline{G}$ is not finitely constrained.
\end{proposition}

\begin{proof}
If a group $G$ satisfies an algebraic law or has a nontrivial center, then so does its closure $\overline{G}$. If $\overline{G}$ is finitely constrained, it must be a regular branch group by Theorem~\ref{t:regular-branch}. It follows from a result in~\cite{Abert-Topological} that a group that satisfies a law can not be a branch group. It is also known that the center of a regular branch group must be trivial~\cite[Theorem 2(c)]{JustInfinite}.
\end{proof}

At present, we do not know if all sofic tree shift groups are finitely constrained, and that question will be addressed in the next section.  

\section{Sofic Tree Shift Groups and Finitely Constrained Groups }

The remainder of this paper will be dedicated to describing sufficient conditions to ensure that a sofic tree shift group is finitely constrained. 

\subsection{Branching Structure and Sofic Tree Shift Groups}

In this part, we outline the structure of certain elements and subgroups of sofic tree shift groups. We begin by introducing the idea of \textit{grafting} one labeled tree onto another (see ~\cite[page 266]{NerodeAutomata})

\begin{definition}
Let $a,b \in A^{X^*}$ and $v \in X^*$. The \textit{grafting} of $b$ onto $a$ at $v$ is the element $g_{[a,b,v]}$ of $A^{X^*}$ given by 
\[ \left( g_{[a,b,v]} \right)_{(w)} = 
 \begin{cases} 
  b_{(u)}, &w=vu \in vX^* \\ 
  a_{(w)}, &w \not\in vX^* 
 \end{cases} 
\] 
\end{definition}

\begin{lemma}[Grafting Lemma]\label{l:grafting}
Let $\mathcal{A}$ be an unrestricted Rabin automaton and suppose $a, b \in \mathcal{L}(\mathcal{A})$ and $v \in X^*$ such that \begin{itemize}
\item[(i)] $a_{(v)} = b_{(\epsilon)}$ \\
\item[(ii)] there exist homomorphisms $\alpha_a: X^* \rightarrow S$ by which $\mathcal{A}$ accepts $a$ and $\alpha_b:X^* \rightarrow S$ by which $\mathcal{A}$ accepts $b$ such that $\alpha_a(v) = \alpha_b(\epsilon)$. 
\end{itemize}
Then $\mathcal{A}$ accepts the grafting of $b$ on $a$ at $v$. 
\end{lemma}

\begin{proof}
Define a map $\alpha_{[a,b,v]}: X^* \rightarrow S$ by 
\[ 
 \alpha_{[a,b,v]}(w) = 
  \begin{cases} 
   \alpha_b(u), &w = vu \in vX^* \\ 
   \alpha_a(w), &w \not\in vX^* 
  \end{cases}. 
\] 
We claim that this is a homomorphism by which $\mathcal{A}$ accepts $g_{[a,b,v]}$.  We must show that for any $w \in X^*$, the transition bundle $\left( \alpha_{[a,b,v]}(w); (g_{[a,b,v]})_{(w)}; (\alpha_{[a,b,v]}(wx))_{x \in X}\right) \in \mathcal{T}$. \\

There are three cases. If $w \in vX^*$ and $w = vu$, then we have
\[ 
\left (\alpha_{[a,b,v]}(w); (g_{[a,b,v]})_{(w)}; (\alpha_{[a,b,v]}(wx))_{x \in X} \right) = \left( \alpha_b(u)); b_{(u)}; (\alpha_b(ux))_{x \in X} \right),  \] 
and $(\alpha_b(u)); b_{(u)}; (\alpha_b(ux))_{x \in X}) \in \mathcal{T}$, since $\alpha_b$ accepts $b$. If $w \not\in vX^*$ and $wx \neq v$ for any $x \in X$, then $\alpha_{[a,b,v]}(w) = \alpha_{a}(w)$, and
\[
 \left( \alpha_{[a,b,v]}(w); (g_{[a,b,v]})_{(w)}; (\alpha_{[a,b,v]}(wx))_{x \in X} \right) = \left( \alpha_a(w); a_{(w)}; (\alpha_a(wx))_{x \in X} \right),
\] 
where $\left( \alpha_a(w); a_{(w)}; (\alpha_a(wx))_{x \in X}) \right) \in \mathcal{T}$, since $\alpha_a$ accepts $a$. Finally, if $v = wx$ for some $x \in X$, then using the fact that $\alpha_{[a,b,v]}(wx) = \alpha_b(\epsilon) = \alpha_a(v)$ gives that 
\[
 \left( \alpha_{[a,b,v]}(w); (g_{[a,b,v]})_{(w)}; (\alpha_{[a,b,v]}(wx))_{x \in X} \right) = \left( \alpha_a(w); a_{(w)}; (\alpha_{a}(wx))_{x \in X} \right) 
\] 
and, again,  $\left( \alpha_a(w); a_{(w)}; (\alpha_{a}(wx))_{x \in X} \right) \in \mathcal{T}$ since $\alpha_a$ accepts $a$. This completes the proof. 
\end{proof} 

\begin{definition}
For $v \in X^*$ and $f \in A^{X^*}$, we denote $g_{[e_G,f,v]}$ by $\delta_{v}(f)$. Note that from the definition, $\delta_{v}(g)$ is given by 
\[ 
 (\delta_{v}(g))_{(w)} = 
  \begin{cases} 
   g_{(z)}, &w = vz \text{ for some } z \in X^* \\ 
   e_A,     &\text{otherwise } 
  \end{cases}  
\] 
\end{definition}

The following useful properties of the $\delta$ operator are easily verified. 

\begin{lemma}
For all $g,h \in A^{X^*}$ and $u, v ,w \in X^*$, the following hold. \begin{itemize}
\item[(i)] $[\delta_{vu}(g)]_{v} = \delta_{u}(g)$
\item[(ii)] $\delta_v(\delta_w(g)) = \delta_{vw}(g)$
\item[(iii)] If $|w| = k$ and $g \in \Triv_{G}(n)$, then $\delta_w(g) \in \Triv_{G}(n+k)$.
\item[(iv)] $\delta_v(gh) = \delta_v(g)\delta_v(h)$ 
\end{itemize}
\end{lemma}

These $\delta$ operators provide the appropriate dynamical viewpoint for a regular branch group in this setting. (Recall the definitions of \textit{symbolic branch} and \textit{regular branch} groups from Subsection~\ref{ss:group-structure}.)

\begin{proposition}\label{p:p-symbolic-branch}
Let $G$ be a subgroup of the full tree shift $A^{X^*}$, and $k \geq 1$. Then $G$ is a \textit{symbolic branch group, symbolically branching over} $\Triv_G(k)$ if $\delta_{x}(g) \in \Triv_{G}(k+1)$ for all $g \in \Triv_{G}(k)$ and $x \in X$. 
\end{proposition}

\begin{proposition}\label{p:alt-regular-branch}
Let $G$ be a subgroup of $\Aut(X^*)$.Then $G$ is a \textit{regular branch group} if it is level-transitive and it is a symbolic branch group over $\Triv_{G}(k)$.
\end{proposition}

Using this notion, Theorem ~\ref{t:regular-branch} may be generalized to the case of self-similar groups which are considered here. The proof is essentially a reproduction of those found in ~\cite{Problems} and ~\cite{HausdorffDim}; we include it in the Appendix for the sake of completeness. 

For the remainder of this subsection, we assume $A$ is a finite group with identity element $e_A$, and let $X$ be a finite alphabet. We also fix an action $\phi$ of $A$ on $X$, and identify the full tree shift group $\mathcal{F}(A,X,\phi)$ with the full tree shift $A^{X^*}$. Also, we let $G$ be a subgroup of $\mathcal{F}(A,X,\phi) = A^{X^*}$ such that the portraits of $G$ form a sofic tree subshift of $A^{X^*}$. The identity of $G$ is denoted by $e_G$. Additionally, we let $\mathcal{A} = (S,X,A,\mathcal{T})$ be an unrestricted Rabin automaton so that $G = \mathcal{L}(\mathcal{A})$, and assume that $|S| = N$. 

\begin{lemma}\label{l:pigeonhole}
If $g \in G$ and $\alpha_g: X^* \rightarrow S$ is a homomorphism by which $\mathcal{A}$ accepts $g$, then there exists an integer $k = k(g)$ satisfying the following conditions:
\begin{itemize}
\item[(i)] for any $w \in X^k$, the restriction of $\alpha$ to the vertices in the path from $\epsilon$ to $w$ is not injective 
\item[(ii)] $|\alpha_g(X^{[k]})| = |\alpha_g(X^{[m]})|$'' for $m \geq k+1$
\item[(iii)] $1 \leq k \leq 2N - 1$.
\end{itemize}
\end{lemma}

\begin{proof}
 Condition (i) is satisfied for any $n \geq N$, by applying the Pigeonhole Principle to the labels of the vertices in the path from $\epsilon$ to a vertex $w \in X^n$. To see (ii), note that the map $\psi: n \mapsto |\alpha(X^{[n])}|$ is a nondecreasing function, bounded above by $N$, with $\psi(N) \geq 1$. Thus, there must be a $k \leq 2N-1$ such that conditions (i) and (ii) are satisfied.
\end{proof}

Henceforth, for $g \in G$, we will write $k(g)$ for the minimal $k$ from Lemma~\ref{l:pigeonhole}.

\begin{lemma}
If $g \in \Triv_G(2N)$ via a homomorphism $\alpha_g$, then there exists a homomorphism $\alpha_e$ that accepts the identity such that $\alpha_e$ and $\alpha_g$ agree on $X^{[k(g)]}$. 
\end{lemma}

\begin{proof}
Let $k = k(g)$. By the previous lemma, $\alpha_g\left(X^{[k]}\right) = \alpha_g\left(X^{[k+1]}\right)$. Note that for all $w \in X^{[k]}$, the transition bundle $(\alpha_g(w); e_A; (\alpha_g(wx))_{x \in X})$ must be in $\mathcal{T}$.  
For each $s \in \alpha_g(X^{[k]})$, choose a word $v \in X^{[k]}$ such that $\alpha_g(v) = s$ and denote $v$ by $\beta_g(s)$. We now define a homomorphism $\alpha_e: X^* \rightarrow S$ by which $\mathcal{A}$ accepts $e_G$. If $w \in X^{[k]}$, set $\alpha_e(w) = \alpha_g(w).$ For $n \geq k+1$, we recursively define $\alpha_e$ on $X^n$ by setting $\alpha_e(wx) = \alpha_g(\beta_g(\alpha_e(w))x)$. 

To see that $\alpha_e$ is a homomorphism by which $\mathcal{A}$ accepts the identity, note that $\alpha_g$ and $\alpha_e$ agree on $X^{[k]}$, and that for all $v$ of length greater than $k$, we must have the transition bundle $(\alpha_e(v); e_A; (\alpha_e(vx))_{x \in X}) \in \mathcal{T}$, since by construction \[ (\alpha_e(v); e_A; (\alpha_e(vx))_{x \in X})) = (\alpha_g(v^*); e_A; (\alpha_g(v^*x))_{x \in X}) \] where $v^* \in X^k$ denotes $\beta_g(\alpha_e(v))$. This completes the proof. 
\end{proof}

\begin{proposition}
Assume $g \in \Triv_{G}(2N)$, $k = k(g)$, $v \in X^{[k]}$, and $u$ is a prefix of $v$. If there exists a homomorphism $\alpha_g: X^* \rightarrow S$ by which $\mathcal{A}$ accepts $g$ such that $\alpha_g(u) = \alpha_g(v)$, then $\delta_v(g_u) \in \Triv_{G}(2N + |v| - |u|)$. 
\end{proposition}

\begin{proof}
Assume $\alpha_g(u) = \alpha_g(v)$. By the previous Lemma, there exists $\alpha_e$ which accepts $e_G$ such that $\alpha_e$ and $\alpha_g$ agree on $X^{[k]}$. Since $G$ is self-similar, there is a homomorphism $\alpha_{g_u}$ by which $\mathcal{A}$ accepts $g_u$ such that $\alpha_{g_u}(\epsilon) = \alpha_g(u)$. (The map $\alpha_{g_u}$ given by $\alpha_{g_u}(w) = \alpha_g(uw)$ is easily seen to be such a homomorphism.) Now we have that 
\[ 
 \alpha_{g_u}(\epsilon) = \alpha_e(u)  = \alpha_e(v) 
\] 
and  
\[ 
 (g_u)_{(\epsilon)} = g_{(u)} =  e_A = (e_G)_{(v)}. 
\] 
Applying the Grafting Lemma yields the desired result. 
\end{proof}

\begin{cor}\label{c:shift}
Assume $g \in \Triv_{G}(2N)$ , $k = k(g)$, and $u, v \in X^{[k]}$ with $u$ a proper prefix of $v$ and $|u| = j$. If there exists a homomorphism $\alpha_g: X^* \rightarrow S$ by which $\mathcal{A}$ accepts $g$ such that $\alpha_g(u) = \alpha_g(v)$, then there exists $u' \in X^{j+1}$ such that $\delta_{u'}(g_u) \in \Triv_G(2N+1)$  \\
\end{cor}

\begin{proof}
By the previous proposition, we know that $\delta_v(g_u) \in \Triv(2N +|v| - |u|)$. Since $v > u$, can write $v = v_1u'$ where $v_1$ is some word (possibly empty) and $|u'| = |u| + 1$. Then we have that $\left[ \delta_v\left(g_u\right) \right]_{v_1} = \delta_{u'}(g_u)$. 
\end{proof}

Recall that for any group, \textit{conjugation} is a right action of the group on itself given by $g^h = h^{-1}gh$. Given $G \leq A^{X^*}$, the \textit{normalizer of G}, denoted by$N_{A^{X^*}}(G)$, consists of the elements of $A^{X^*}$ which leave $G$ fixed under conjugation; i.e.,
\[ 
N_{A^{X^*}}(G) = \{ h \in A^{X^*} \; \mid \; g^h \in G \text{ for all } g \in G \}. 
 \]

The following lemma is proven in ~\cite[Lemma 1]{PatternClosure} for self-similar groups of tree automorphisms. We will not reproduce the proof here, since it is a lengthy computation which is easily generalized to our current setting. 

\begin{lemma}\label{l:movement}
Let $g, h \in A^{X^*}$ and $u \in X^*$. Then $(\delta_u(h))^{g} = \delta_v(h^{(g_v)})$, where $v = g^{-1}(u)$. 
\end{lemma}

We use this Lemma in the following Proposition.

\begin{proposition}\label{p:movement}
Let $G$ be a subgroup of $A^{X^*}$ such that $N_{A^{X^*}}(G)$ contains a self-similar, self-replicating, level-transitive subgroup. If $\delta_u(g) \in G$ for some $g \in G$ and $u \in X^n$, then $\delta_v(g) \in G$ for all $v \in X^n$. 
\end{proposition}

\begin{proof}
Suppose $\delta_u(g) \in G$ for some $u \in X^n$ and $g \in G$. Let $v \in X^n$ be arbitrary. Let $N = N_{A^{X^*}}(G)$ be the normalizer of $G$, and assume that $N$ contains a self-similar, self-replicating, level-transitive subgroup $M$. Since $M$ is level-transitive, there exists $f \in M$ such that $f(v) = u$. Since $M$ is self-similar, $(f_v)^{-1} \in M$, and since $M$ is self-replicating, there exists $f' \in \Stab_{M}(v)$ such that $(f')_v = (f_v)^{-1}$. Then $(\delta_u(g))^{ff'} \in G$ since $M$ normalizes $G$. Moreover, from these observations and Lemma \ref{l:movement}, it follows that \begin{alignat*}{2}
\left(\delta_u(g) \right)^{ff'} &= \left( (\delta_u(g))^f\right)^{f'} \\ 
&= \left( \delta_{v}(g^{f_v})  \right)^{f'} \qquad && \text{ since } v = f^{-1}(u) \\
&= \delta_{v}\left(g^{f_v(f_v)^{-1}}\right) \qquad && \text{ since } v = (f')^{-1}(v) \\
&= \delta_{v}(g) 
\end{alignat*}
\end{proof}

\subsection{Main Results}

We now have the necessary framework in place to prove our desired results. 

\begin{theorem}\label{t:main}
Let $G$ be a subgroup of $A^{X^*}$. If $N_{A^{X^*}}(G)$ contains a self-similar, self-replicating, level-transitive subgroup, then $G$ is a sofic tree shift group if and only if $G$ is a finitely constrained group. 
\end{theorem}

\begin{proof}
Since every finitely constrained group is a sofic tree shift group, we only need to prove one direction. Since $G$ is a sofic tree shift group, there exists an unrestricted Rabin automaton $\mathcal{A}_G$ such that $\mathcal{L}\left( \mathcal{A} \right) = G$. Assume that $\mathcal{A}_G$ has a state set $S$ such that $|S| = N$. We will prove that $G$ is a regular branch group over the subgroup $\Triv_G(2N)$. Let $g \in \Triv_G(2N)$ and $\alpha_g: X^* \rightarrow S$ be a homomorphism by which $\mathcal{A}$ accepts $g$. Let $k = k(g)$.

For $w \in X^k$, let $\mu(w)$ be the shortest prefix of $w$ such that the state $\alpha_g(\mu(w))$  is repeated in the path from $\epsilon$ to $w$. Let 
\[ B_g = \{ \mu(w) \; \mid \; w \in X^k \}, \] 
and construct a set $C_g$ from $B_g$ as follows: if $b,b' \in B_g$ with $b$ a proper prefix of $b'$, remove $b'$. It is clear that after the inevitable termination of this procedure, $C_g$ satisfies the following conditions

\begin{itemize}
\item[(i)] for any $u \in X^k$, there is a prefix of $u$ in $C_g$ 
\item[(ii)] no word in $C_g$ is a proper prefix of another word in $C_g$
\end{itemize} 

Then, for any $g \in \Triv_G(2N)$, we can write $\displaystyle{g = \prod_{c \in C_g} \delta_c(g_c)}$. For distinct elements $c, c' \in C_g$, the elements $\delta_c(g_c)$ and $\delta_{c'}(g_{c'})$ commute, as they are supported on disjoint subtrees. Let $x \in X$. By Corollary~\ref{c:shift} and Proposition~\ref{p:movement}, the element $\delta_{xc}(g_c) \in \Triv_G(2N+1)$ for all $c \in C_g$. Hence $\Triv_G(2N+1)$ also contains the product $\displaystyle{\prod_{c \in C_g}\delta_{xc}(g_c)}$. Further, we note that 
\begin{align*} 
\prod_{c \in C_g} \delta_{xc}(g_c) 
 &= \prod_{c \in C_g}  \delta_x(\delta_c(g_c)) \\
 &= \delta_x\left(\prod_{c \in C_g} \delta_c(g_c)) \right) \\
 &= \delta_x(g).
\end{align*}  
Therefore $G$ is a symbolic branch group over the subgroup $\Triv_{G}(2N)$, and $G$ is finitely constrained. \\
\end{proof}

\begin{cor}
The closure of the odometer is not a sofic tree shift group.  
\end{cor}

\begin{proof}
Let $\mathcal{O}$ represent the odometer group. Since $\mathcal{O}$ is a self-similar, self-replicating, level-transitive subgroup of $\overline{\mathcal{O}}$ and we have shown that $\overline{\mathcal{O}}$ is not finitely constrained, this result follows immediately from Theorem \ref{t:main}.
\end{proof}

\begin{proposition}\label{p:algebraic-law-not-sofic}
Let $G$ be a self-similar,  level-transitive group of tree automorphisms such that $G$, or its normalizer in $\Aut(X^*)$, contains a self-replicating, level-transitive subgroup. If $G$ satisfies an algebraic law, or $G$ has a nontrivial center, then the topological closure $\overline{G}$ is not Rabin-recognizable. 
\end{proposition}

\begin{proof}
 Note that $\overline{G}$ is self-similar, level-transitive, and topologically closed. Suppose $\overline{G}$ is Rabin-recognizable. Then it would be a sofic tree shift by~\cite[Theorem 1.7]{SoficTreeShiftCA}, and hence, by Theorem~\ref{t:main}, a finitely constrained group.  By Proposition~\ref{p:algebraic-law-not-finitely-constrained}, this implies that $G$ can not obey an algebraic law or have a nontrivial center. 
\end{proof}

\appendix

\section{Characterization of Generalized Finitely Constrained Groups} 

\begin{theorem}\label{t:appendix}
Let $G$ be a subset of the full tree shift $A^{X^*}$ and $n \geq 1$. The following are equivalent. 

(i) $G$ is a finitely constrained group defined by patterns of size $s$. 

(ii) $G$ is a closed, self-similar group branching symbolically over $\Triv_G(s - 1)$. 

(iii) $G$ is the closure of a self-similar group $H$ which is branching symbolically over $\Triv_H(s - 1)$. 
\end{theorem}

\begin{proof}
$(i) \implies (ii).$
Let $G$ be a finitely constrained group defined by a set $\mathcal{F}$ of forbidden patterns of size $s$. Since $G$ is a tree shift group, it is closed and self-similar. Let $g \in \Triv_{G}(s-1)$ and $x \in X$. Observe that $\delta_x(g)$ is an element of $\Triv(s)$, and we must show it is also in $G$. Since $\delta_x(g) \in \Triv(s)$, the pattern of size $s$ appearing at the root in $\delta_x(g)$ is trivial (all labels are equal to $e_A$). The first level sections of $\delta_x(g)$ are either $g$ or trivial, so no patterns from $\mathcal{F}$ appear in any section of $\delta_x(g)$. As no forbidden patterns occur anywhere in $\delta_x(g)$, it follows that $\delta_x(g) \in G$. Thus $G$ is branching symbolically over $\Triv_G(s-1)$. 

$(ii) \implies (i)$. Let $G$ be a closed, self-similar group branching symbolically over $\Triv_{G}(s-1)$. We claim that $G$ is branching symbolically over $\Triv_{G}(n)$ for every $n \geq s-1$. Indeed, for $n \geq s-1$, if $g \in \Triv_{G}(n)$, then $g \in \Triv_{G}(s-1)$, which implies that $\delta_x(g) \in \Triv_{G}(s-1)$, for $x \in X$. In particular, $\delta_x(g) \in G$, for $x \in X$, which implies that $\delta_x(g) \in \Triv_G(n)$.

Let $\mathcal{A}$ be the set of patterns of size $s$ that appear in at least one element of $G$ at least one vertex. By self-similarity, these are exactly the same patterns of size $s$ that appear in the elements of $G$ at the root. Let $\mathcal{F}$ be the set of patterns of size $s$ that are not in $\mathcal{A}$ and $G_{\mathcal{F}}$ be the set of all portraits that do not contain any patterns from $\mathcal{F}$. The set $G_{\mathcal F}$ is closed and self-similar, since it is a tree shift. By the definition of $\mathcal{F}$, we have $G \subseteq G_{\mathcal{F}}$. 

It remains to show that $G_{\mathcal{F}} \subseteq G$, which would show that $G=G_{\mathcal{F}}$ and, therefore, $G$ is a finitely constrained group defined by patterns of size $s$.  

For every $g \in G_{\mathcal{F}}$, we will construct a sequence $\{g_n\}_{n=0}^\infty$ of elements in $G$ such that the portraits of $g_n$ and $g$ agree on $X^{[n]}$, for $n \geq 0$. Since $G$ is closed this would show that $G_{\mathcal{F}} \subseteq G$, completing the proof.  

We proceed by induction on $n$ and, simultaneously, all $g \in G_{\mathcal{F}}$. Since $g \in G_{\mathcal{F}}$ avoids the forbidden patterns, its pattern of size $s$ at the root coincides with the pattern of size $s$ of some element $g_0 \in G$ at the root. Thus, we may set $g_n = g_0$, for $n=0,\dots,s-1$. Assume that, for some fixed $n \geq s-1$ and every $g \in G_{\mathcal{F}}$, there exists an element $g_n \in G$ such that the portraits of $g_n$ and $g$ agree on $X^{[n]}$. Let $f=g_n^{-1}g$. Since $g_n \in G \subseteq G_{\mathcal{F}}$, we have that $f \in G_{\mathcal{F}}$. By the self-similarity of $G_{\mathcal{F}}$, the sections $f_x$, for $x \in X$, are also in $G_{\mathcal{F}}$. By the inductive hypothesis, there exist elements $f^{(x)} \in G$, for $x \in X$, such that the portraits of $f^{(x)}$ and $f_x$ agree on $X^{[n]}$ (the superscript in $f^{(x)}$ is for indexing purposes). Since $g_n$ and $g$ agree on $X^{[n]}$, we have $f = g_n^{-1}g \in \Triv(n+1)$ and, consequently, $f_x \in \Triv(n)$, for $x \in X$. Since $f_x$ and $f^{(x)}$ agree on $X^{[n]}$ and $f^{(x)} \in G$, we obtain that $f^{(x)} \in \Triv_G(n)$, for $x \in G$. But $n \geq s-1$ and $G$ is symbolically branching over $\Triv_G(n)$, which implies that $f' = \prod_{x \in X} \delta_x(f^{(x)}) \in \Triv_G(n)$. Thus, $f'$ is an element of $G$ and the portraits of $f'$ and $f=g_n^{-1}g$ agree on $X^{[n+1]}$, implying that the portraits of $g_nf'$ and $g$ agree on $X^{[n+1]}$. Thus, we may set $g_{n+1} = g_nf'$ and, since $g_nf' \in G$, this completes the induction step.  

$(ii) \implies (iii)$. Set $H=G$. 

$(iii) \implies (ii)$. The set $G$ is a self-similar group since the closure of a subgroup of $A^{X^*}$ is a subgroup and the closure of a self-similar set is self-similar (both of these follow from the fact that sections of products are products of sections). For all $n$, a convergent sequence in $A^{X^*}$ converges to an element in $\Triv(n)$ if and only if its members are eventually in $\Triv(n)$. Since $H$ is branching symbolically over $\Triv_H(s-1)$, its closure $G$ is branching symbolically over $\Triv_G(s-1)$.

\end{proof}

\end{document}